\documentclass[11pt]{amsart}
\usepackage{amscd, amssymb}
\usepackage{graphics}

\theoremstyle{plain}

\numberwithin{equation}{section}
\newtheorem{theorem}[equation]{Theorem}
\newtheorem{proposition}[equation]{Proposition}
\Huge
\newtheorem{lemma}[equation]{Lemma}

\theoremstyle{definition}
\newtheorem{remark}[equation]{Remark}
\newtheorem{example}[equation]{Example}

\def    \R  {{\Bbb R}}
\def    \Z  {{\Bbb Z}}
\def    \Q  {{\Bbb Q}}

\def    \CP {{\Bbb {CP}}}
\def    \P {{\Bbb {P}}}
\def    \C  {{\Bbb C}}
\def    \N  {{\Bbb N}}

\def  \rank {{\operatorname{rank}}}
\def    \codim  {{\operatorname{codim}}}

\def  \deg {{\operatorname{deg}}}

\def   \lot   {{\operatorname{\ lower \ order\ terms }}}   
\def    \Tilde  {\widetilde}
\def    \ut     {\Tilde{u}}

\def    \Gt     {\Tilde{G}}

\begin{document}
\title[Hamiltonian circle actions with fixed point set almost minimal] {Hamiltonian circle actions with  fixed point set almost minimal}

\author{Hui Li}
\address{School of mathematical Sciences\\
               Soochow University \\
               Suzhou, 215006, China.}
        \email{hui.li@suda.edu.cn}

\thanks{2000 classification. 53D05, 53D20, 55N25, 57R20, 32H02} 

\begin{abstract}
Motivated by recent works on Hamiltonian circle actions  satisfying certain minimal conditions,  in this paper, we consider Hamiltonian circle actions satisfying an almost minimal
condition. More precisely, we consider a compact symplectic manifold $(M, \omega)$ admitting a Hamiltonian
circle action with fixed point set consisting of two connected components $X$ and $Y$ satisfying 
$\dim(X)+\dim(Y)=\dim(M)$. Under certain cohomology conditions,  we determine the circle action, the integral cohomology rings of $M$, $X$ and $Y$, and the total Chern classes of $M$, $X$, $Y$, and of the normal bundles of $X$ and $Y$. The results show that 
these data are unique --- they are exactly the same as those in the standard example $\Gt_2(\R^{2n+2})$, the Grassmannian of oriented $2$-planes in $\R^{2n+2}$, which is of dimension $4n$ with (any) $n\in\N$, equipped with a standard circle action. Moreover, if $M$ is   
K\"ahler and the action is holomorphic, we can use a few different criteria to claim that $M$ is $S^1$-equivariantly biholomorphic and $S^1$-equivariantly symplectomorphic to $\Gt_2(\R^{2n+2})$. 
\end{abstract}

 \maketitle

\section{Introduction}
Let the circle act effectively on a compact symplectic manifold $(M,\omega)$ of dimension $2n$ with moment map
$\phi\colon M\to\R$. The fixed point set $M^{S^1}$, which is the same as the critical set of $\phi$,
contains at least two connected components --- the minimum and the maximum of $\phi$. In the case
when $M^{S^1}$ consists of exactly two connected components, $X$ and $Y$, since $M$ is compact and symplectic, by Morse-Bott theory, we must have
$$\dim(X)+\dim(Y) +2\geq \dim(M).$$
In \cite{LT}, we studied the case when the following {\it minimal} condition holds:
\begin{equation}\label{=}
\dim(X)+\dim(Y) +2 = \dim(M).
\end{equation}
Standard examples of manifolds satisfying (\ref{=}) are $\CP^n$, and $\Gt_2(\R^{n+2})$, the
Grassmannian of oriented $2$-planes in $\R^{n+2}$ with $n >1$ odd, equipped with standard circle actions. For the case when (\ref{=}) holds, we classified the circle action, the integral cohomology rings of $M$, $X$ and $Y$, and the total Chern classes of $M$, $X$, $Y$, and of the normal bundles of $X$ and $Y$. It turns out that these data on $M$ are exactly as in the two families of standard examples. In particular, $b_{2i}(M)=1$ and $b_{2i-1}(M)=0$ for all $0\leq 2i\leq\dim(M)$.
    
Note that for a compact symplectic manifold $(M, \omega)$, $0\neq [\omega]^i\in H^{2i}(M; \R)$, so
we always have $b_{2i}(M)\geq 1$ for all $0\leq 2i\leq\dim(M)$. If $b_{2i}(M) = 1$ for some  $2i$, we say that the $2i$-th Betti number of $M$ is {\it minimal}. The even Betti numbers of $\CP^n$ and of $\Gt_2(\R^{n+2})$ with $n > 1$ odd are all minimal.

Originally motivated by the classical Petrie's conjecture, Tolman studied Hamiltonian
circle actions on compact $6$-dimensional manifolds with minimal even Betti numbers \cite{T}. Then a number of work
\cite{{M}, {LT}, {LOS}, {L}, {Mor}, {GS}, {JT}} appeared about Hamiltonian circle actions on compact manifolds with minimal even Betti numbers or with fixed point set satisfying a minimal condition 
$$\sum_{F\subset M^{S^1}}(\dim(F)+2) = \dim(M)+2,$$ 
where the $F$'s are connected components of the fixed point set. These two minimal conditions are closely related \cite[Sec. 4]{LT}. In particular, the condition of having minimal even Betti numbers implies the condition of fixed point set satisfying the above equality.

 We know an example, $\Gt_2(\R^{2n+2})$, of dimension $4n$, equipped with a standard Hamiltonian circle action, whose fixed point set consists of two connected components $X$ and $Y$ satisfying the following {\it almost minimal}
condition
\begin{equation}\label{a=}
\dim(X)+\dim(Y) = \dim(M).
\end{equation}
Note that $b_{2i}\big(\Gt_2(\R^{2n+2})\big)=1$ for $0\leq 2i\leq 4n$ and
$2i\neq 2n$, $b_{2n}\big(\Gt_2(\R^{2n+2})\big)=2$, and $b_{2i-1}\big(\Gt_2(\R^{2n+2})\big)=0$ for all $i$. So the even Betti numbers of $\Gt_2(\R^{2n+2})$ are almost minimal.
   
 In this paper, we look at Hamiltonian circle actions on compact manifolds with fixed point set consisting of two connected components satisfying (\ref{a=}). Let us first  look at the standard example and the data on it.

\begin{example}\label{grass}
Let $n\geq 1$, and let $\Gt_2(\R^{2n+2})$ be the Grassmannian of oriented $2$-planes in $\R^{2n +2}$.
This $4n$-dimensional manifold is a coadjoint orbit of $SO(2n+2)$, so it is a symplectic
(K\"ahler) manifold and it admits a Hamiltonian $SO(2n+2)$ action. 

There is a Hamiltonian $S^1\subset SO(2n+2)$ action on $\Gt_2(\R^{2n+2})$ induced by the $S^1$ action on
$\R^{2n+2}\cong \C^{n+1}$ given by
$$\lambda\cdot (z_1, \cdots, z_{n+1}) = (\lambda z_1, \cdots, \lambda z_{n+1}).$$
The fixed point set consists of two connected components $X$ and $Y$, where 
$$X\cong Y\cong \CP^n,$$
corresponding to the two orientations on the real $2$-planes in $\P(\C^{n+1})$.  The $S^1/\Z_2\cong S^1$ action on 
$\Gt_2(\R^{2n+2})$
is {\it semifree} (i.e., free outside fixed points). 

  Let $[\omega]$ be a primitive integral class on $\Gt_2(\R^{2n+2})$ represented by the symplectic form $\omega$ such that  $[\omega|_X]=u$ and $[\omega|_Y]=v$ are
also primitive integral (such a form exists by Lemma~\ref{prim} when $n\geq 2$. If $n=1$, $\Gt_2(\R^{2n+2})$ is diffeomorphic to $S^2\times S^2$, we take $\omega=x_1+x_2$, where $x_1$ and $x_2$ are positive generators
of $H^2(S^2\times\mbox{pt}; \Z)$ and $H^2(\mbox{pt}\times S^2; \Z)$ respectively.). Then the total Chern classes of 
$\Gt_2(\R^{2n+2})$, $X$, $Y$, and of the normal bundles $N_X$ of $X$ and $N_Y$ of $Y$ are respectively as follows: 
$$c\big(\Gt_2(\R^{2n+2})\big) = \frac{(1+[\omega])^{2n+2}}{1+2[\omega]},\,\,\,\mbox{in particular},\,\,\, c_1\big(\Gt_2(\R^{2n+2})\big) = 2n[\omega],$$
$$c(X) = (1+u)^{n+1}, \,\,\,  c(Y) = (1+v)^{n+1},$$
$$c(N_X)=\frac{(1+u)^{n+1}}{1+2u}, \,\,\mbox{and}\,\,\,  c(N_Y) = \frac{(1+v)^{n+1}}{1+2v}.$$
The integral cohomology groups of $\Gt_2(\R^{2n+2})$ are:
 \[  H^i\big(\Gt_2(\R^{2n+2}); \Z\big) = \left\{ \begin{array}{ll}
           \Z    &  \mbox{if $ i$  is even and  $i\neq 2n$},\\
           \Z\oplus \Z  &    \mbox{if $i =2n$},\\
           0    & \mbox{if $i$  is odd}.
           \end{array} \right. \]
 The integral cohomology ring of  $\Gt_2(\R^{2n+2})$ is
\[ H^*\big(\Gt_2(\R^{2n+2}); \Z\big) = \left\{\begin{array}{ll}
 \Z[x, y]/\big(x^{n+1}-2xy, y^2\big) & \mbox{if $n$ is odd},\\
\Z[x, y]/\big(x^{n+1}-2xy, y^2-x^ny\big) & \mbox{if $n$ is even},
\end{array}\right.\]
where $x=[\omega]$ and $\deg(y)=2n$.

These data also follow from the results of this paper. In the literature, the cohomology ring 
$H^*\big(\Gt_2(\R^{2n+2}); \Z\big)$ may be given differently using different generators.
\end{example}

For the case when (\ref{=}) holds, the condition (\ref{=}) implies that the even degree cohomology groups
of $X$ and $Y$ are one dimensional, for the current case when (\ref{a=}) holds, this is not true, see Remark~\ref{mam}. In this paper, for the case when (\ref{a=}) holds, the idea is we assume one or two fixed point components have the same {\it even} degree integral cohomology as some complex projective spaces (just as in Example~\ref{grass}), we show that the important data on $M$ are exactly as in Example~\ref{grass}. These results, in particular, imply that in the K\"ahler case, $M$ is unique up to $S^1$-equivariant biholomorphism and symplectomorphism.

\begin{remark}
In our main theorems, for $\dim(M) > 4$, our assumption grants that $\dim H^2(M; \R)=1$, so up to scaling, we can assume the symplectic class $[\omega]$ is primitive integral. This is not a serious assumption.
\end{remark}

Now let us state our main results. First, the first Chern class $c_1(M)$ is an important data, in particular for the K\"ahler case. For the two fixed point set
components $X$ and $Y$, with no loss of generality, we may assume 
$\dim(X)\geq\frac{1}{2}\dim(M)$. With a condition on $H^{\mbox{even}}(X; \Z)$, we obtain a result on the action and on $c_1(M)$ as follows. 

\begin{theorem}\label{c1}
Let $(M, \omega)$ be a compact symplectic manifold of dimension bigger than $4$ admitting an effective Hamiltonian $S^1$ action with moment map $\phi$ such that the fixed point set consists of two connected components $X$ and $Y$ with $\dim(X)+\dim(Y)=\dim(M)$. Assume $[\omega]$ is a primitive integral class, $\dim(X)\geq\frac{1}{2}\dim(M)$, and $H^{\mbox{even}}(X; \Z)=\Z[u]/u^{\frac{1}{2}\dim(X)+1}$, where $u=[\omega|_X]$. Then $\dim(X) = \frac{1}{2}\dim(M)$ and the following $3$ conditions are equivalent:
\begin{enumerate}
\item the action is semifree,
\item $|\phi(Y)-\phi(X)|=1$, \,\,\mbox{and}
\item $c_1(M)= \frac{1}{2}\dim(M) [\omega]$.
\end{enumerate}
\end{theorem}

Assuming a condition on both fixed point set components, we obtain the following results. 

\begin{theorem}\label{lg}
Let $(M, \omega)$ be a compact symplectic manifold of dimension bigger than $4$ admitting an effective Hamiltonian $S^1$ action such that the fixed point set consists of two connected components $X$ and $Y$ with $\dim(X)+\dim(Y)=\dim(M)$. Assume $[\omega]$ is a primitive integral class,
$H^{\mbox{even}}(X; \Z) = \Z[u]/u^{\frac{1}{2}\dim(X)+1}$ and $H^{\mbox{even}}(Y; \Z)=\Z[v]/v^{\frac{1}{2}\dim(Y)+1}$, where $u= [\omega|_X]$
and $v= [\omega|_Y]$. Then 
\begin{enumerate}
\item $\dim(X)=\dim(Y)=2n$, $\dim(M)=4n$, where $n >1$,
\item the action is semifree,
\item $H^*(X; \Z)= \Z[u]/u^{n+1}$  and  $H^*(Y; \Z)=\Z[v]/v^{n+1}$,
\item $H^*(M; \Z)\cong H^*\big(\Gt_2(\R^{2n+2}); \Z\big)$ as rings,
\item $c(M)\cong c\big(\Gt_2(\R^{2n+2})\big)$, and
\item $c(X)$, $c(Y)$, $c(N_X)$ and $c(N_Y)$ are all isomorphic to those in Example~\ref{grass}.
\end{enumerate}
\end{theorem}
\begin{remark}\label{d4}
Theorems~\ref{c1} and \ref{lg} do not hold for dimension $4$. For example, take 
$(M, \omega)=(\Sigma_g\times S^2, x_1 + 2x_2)$, where $\Sigma_g$ is a surface of genus $g\geq 0$, 
$x_1$ is a positive $H^2(\Sigma_g; \Z)$ generator and $x_2$ is a positive $H^2(S^2; \Z)$ generator.
Let $S^1$ act on $M$ by fixing $\Sigma_g$ and rotating semifreely on $S^2$. Then $X\cong Y\cong \Sigma_g$,
and the moment map image has length $2$.

Let $(M, \omega)$ be $4$-dimensional and satisfies the other conditions
of Theorem~\ref{lg}. Then the following things are true:
\begin{itemize}
\item [(a)]  $\dim(X)=\dim(Y)=2$, and $X\cong Y$, (where $\cong$ means diffeomorphic.)
\item [(b)]  the action is semifree, and
\item [(c)]  $c(N_X)=c(N_Y)=1$.
\end{itemize}
In fact, since the action is effective and $X$ and $Y$ are even
dimensional, $\dim(X)=\dim(Y)=2$; (b) easily follows; $X\cong Y$ follows  by looking at the symplectic reduced spaces at all values of the moment map; (c) is by Lemma~\ref{bothprim}.
If we pose the stronger assumption $H^*(X; \Z)=\Z[u]/u^2$, and still
assume $[\omega|_Y]$ is primitive integral, then 
\begin{itemize}
\item [(d)]  $X\cong Y\cong S^2$, and
\item [(e)]  $M$ is $S^1$-equivariantly diffeomorphic to $S^2\times S^2$, where $S^1$ acts on $S^2\times S^2$ by fixing the first sphere and rotating semifreely on the second sphere.
\end{itemize}
The claim (e) follows from (c) and (d).
If we choose $\omega$ suitably such that $[\omega|_{\mbox{pt}\times S^2}]$ is primitive integral, then all the data are exactly as in Example~\ref{grass} (for dimension $4$), and $|\phi(Y)-\phi(X)|=1$, where $\phi$ is the moment map.
\end{remark}

If we assume $H^*(M; \Z)\cong H^*\big(\Gt_2(\R^{2n+2}); \Z\big)$ as rings, we can determine all the other data for all dimensions.

\begin{theorem}\label{gl}
Let $(M, \omega)$ be a compact symplectic manifold admitting an effective 
Hamiltonian $S^1$ action  such that the fixed point set consists of two connected components $X$ and $Y$ with 
$\dim(X)+\dim(Y)=\dim(M)$. Assume $H^*(M; \Z)\cong H^*\big(\Gt_2(\R^{2n+2}); \Z\big)$ as rings. Then
\begin{enumerate}
\item $H^*(X; \Z)\cong H^*(Y; \Z)\cong H^*(\CP^n; \Z)$ as rings,
\item the action is semifree, 
\item $c(M)\cong c\big(\Gt_2(\R^{2n+2})\big)$, and
\item $c(X)$, $c(Y)$, $c(N_X)$ and $c(N_Y)$ are all isomorphic to those in Example~\ref{grass}.
\end{enumerate}
\end{theorem}

Let $(M, \omega, J)$ be a compact K\"ahler manifold of complex dimension $n$, and assume $[\omega]$ is an integral class. Then $c_1(M)=n[\omega]$ implies that $M$ is biholomorphic to $\Gt_2(\R^{n+2})$ \cite{KO}. Our main results Theorems~\ref{c1}, \ref{lg} and \ref{gl} and the method in \cite{L} allow us to use various criteria to identify the following Hamiltonian $S^1$-K\"ahler manifold with Example~\ref{grass} in the $S^1$-equivariant complex and symplectic categories.

\begin{theorem}\label{K}
Let $(M, \omega, J)$ be a compact K\"ahler manifold of  complex dimension $2n$ admitting an effective 
holomorphic Hamiltonian $S^1$ action with moment map $\phi$ such that the fixed point set consists of two connected components $X$ and $Y$ with $\dim(X)+\dim(Y)=\dim(M)$. Assume $[\omega]$ is a  primitive  integral class (suitably chosen if $\dim_{\C}(M)=2$).
Then any one of the following conditions implies that
$M$ is $S^1$-equivariantly biholomorphic and symplectomorphic to $\Gt_2(\R^{2n+2})$ as in Example~\ref{grass}.
\begin{enumerate}
\item $c_1(M) = \dim_{\C}(M) [\omega]$.
\item $H^*(M; \Z)\cong H^*\big(\Gt_2(\R^{2n+2}); \Z\big)$ as rings.
\item $H^*(X; \Z)=\Z[u]/u^{\dim_{\C}(X)+1}$  and $H^*(Y; \Z)=\Z[v]/v^{\dim_{\C}(Y)+1}$,
where $u= [\omega|_X]$ and $v= [\omega|_Y]$.
\item $\dim_{\C}(M) =2n > 2$, $H^{\mbox{even}}(X; \Z)=\Z[u]/u^{\dim_{\C}(X)+1}$  and 
$H^{\mbox{even}}(Y; \Z)=\Z[v]/v^{\dim_{\C}(Y)+1}$.
\item $\dim_{\C}(M) =2n > 2$, $H^{\mbox{even}}(X; \Z)=\Z[u]/u^{n+1}$, and the action is semifree.
\item $\dim_{\C}(M) =2n > 2$, $H^{\mbox{even}}(X; \Z)=\Z[u]/u^{n+1}$, and $|\phi(Y)-\phi(X)|=1$.
\end{enumerate}
\end{theorem}

In Theorem~\ref{K}, $(1)$ is the condition we will use to make the claim.
The conditions $(1)$, $(2)$ and $(3)$ allow us to make the claim 
for any dimensions.  
By Theorem~\ref{gl}, for the same dimensional manifolds, $(2)$ implies all the other conditions. Clearly, $(3)$ implies $(4)$ for dimension bigger than $4$, and $(3)$ implies $(1)$ for dimension $4$ by Remark~\ref{d4}. By Theorem~\ref{lg}, $(4)$ implies 
$(1)$ and $(5)$ for dimension bigger than $4$. By Theorem~\ref{c1}, $(5)$ and $(6)$ are equivalent, and they imply $(1)$ for dimension bigger than $4$.
  
\begin{remark}\label{mam}
For the case when (\ref{=}) holds, in \cite[Prop. 4.2]{LT}, we showed that condition (\ref{=}) implies that
the even degree cohomology groups of $X$ and $Y$ are all one dimensional. This is not true for the case when
(\ref{a=}) holds. A counter example is:  take $M=\CP^1\times\CP^n$, and let $S^1$ act on $M$ by
$$\lambda\cdot ([z_0, z_1]\times [w_0, w_1, \cdots, w_n])=[z_0, z_1]\times [w_0, \lambda w_1, \cdots, \lambda w_n].$$
The two fixed point components are $X=\CP^1\times\mbox{pt}$ and $Y=\CP^1\times\CP^{n-1}$ satisfying (\ref{a=}).
Hence condition (\ref{a=}) itself does not allow us to weaken the assumptions of our theorems.
\end{remark}

The organization of the paper is as follows. In Section~\ref{P}, we give some preliminary
results for the next sections. In Section~\ref{thm1}, we prove Theorem~\ref{c1}. 
In Sections~\ref{action} -- \ref{sec-thm2}, we assume the condition on the even degree cohomology groups of both fixed point set components as in Theorem~\ref{lg}.  
In Section~\ref{action}, we use
the condition to rule out non-semifree actions. In Section~\ref{sec-euler}, we determine the equivariant Euler classes of the normal bundles of the fixed point set components. This is an important step for the next two sections. In Section~\ref{sec-ring}, we determine the integral cohomology
rings of the fixed point set components. In Section~\ref{sec-totalchern}, we obtain the total Chern classes of the fixed point sets and
of their normal bundles. In Section~\ref{sec-thm2}, we obtain the integral cohomology ring and total Chern class of the
manifold $M$ and prove Theorem~\ref{lg}. In Section~\ref{sec-thm3}, we prove Theorem~\ref{gl}, and 
in Section~\ref{sec-thm4}, we consider the K\"ahler case and prove Theorem~\ref{K}.

\subsubsection*{Acknowledgement}  
The author would like to thank the referee for some comments which help to improve the exposition.

    This work is supported by the NSFC grant K110712116.

\section{some preliminaries}\label{P}
In this section, we state and prove some preliminary results which we will use in the next sections. 
In this paper, we will use equivariant cohomology techniques. We refer to \cite{LT} for  
the basic material and summary of useful facts about $S^1$-equivariant cohomology. 

\medskip

First, let us set up some notations:
\begin{enumerate}
\item  $H^*_{S^1}(M; R)$ --- the $S^1$-equivariant cohomology of the $S^1$-manifold $M$ with coefficient ring $R$.
\item $t$ ---  a generator of $H^2_{S^1}(\mbox{pt}; \Z)=H^2(\CP^{\infty}; \Z)$.
\item $N_X$ --- the normal bundle of a submanifold $X$ in a manifold. 
\item $e^{S^1}(N_F)$ --- the $S^1$-equivariant Euler class of the normal bundle $N_F$ of a fixed point set component 
         $F$ in an $S^1$-manifold.
\item $c^{S^1}(N_F)$ --- the $S^1$-equivariant total Chern class of the normal bundle $N_F$ of a fixed point set component $F$ in an $S^1$-manifold.
\item $c^{S^1}(M)$ --- the $S^1$-equivariant total Chern class of the $S^1$-manifold $M$.
\item $\Gamma_F$ --- the sum of the weights of the $S^1$ action on the normal bundle of the fixed point set component $F$ in an $S^1$-manifold.
\end{enumerate}

The following elementary fact is essential in applications.
\begin{lemma}\label{ut}
Let the circle act  on a compact symplectic manifold $(M,\omega)$
with moment map $\phi \colon M \to \R$, such that $M^{S^1}$ consists
of two connected components $X$ and $Y$. Assume $[\omega]$ is an integral class. 
Then there exists $\ut  \in H_{S^1}^2(M;\Z)$ such that 
$$\ut|_X = [\omega|_X],\,\,\mbox{and}\,\,\,    \ut|_Y = [\omega|_Y]  + \big(\phi(X)-\phi(Y)\big) t.$$
In particular,  $\phi(X) - \phi(Y) \in \Z$. Moreover, if $\Z_k\subsetneq S^1$ is the stabilizer group of some point on $M$,  
then $k\,|\left(\phi(X) - \phi(Y)\right)$.
\end{lemma}

\begin{proof}
For the existence of  $\ut$, see \cite[Lemma 2.7]{LT}. If $\Z_k$ is the stabilizer group of some point, then the submanifold $M^{\Z_k}$ fixed by 
$\Z_k$ contains $X$ and $Y$, and the $S^1/\Z_k\cong S^1$ action on  $M^{\Z_k}$ has moment map
$\phi' = \frac{\phi}{k}$. Apply the first claim on $M^{\Z_k}$ for the $S^1/\Z_k\cong S^1$ action, 
we get $\phi'(X)-\phi'(Y)\in\Z$, which means $k\,|\left(\phi(X) - \phi(Y)\right)$.
\end{proof}

Using Lemma~\ref{ut}, we get the following result on $e^{S^1}(N_X)$, which will be an important tool  in our proofs.

\begin{lemma}\label{mult}
Let $(M, \omega)$ be a compact symplectic manifold admitting a Hamiltonian $S^1$ action with moment map $\phi$ such that $M^{S^1}$ consists of two connected components $X$ and $Y$. Assume $[\omega]$ is an integral class.
Then there exists $\lambda\in H^*_{S^1}(X; \Z)$ such that
$$\lambda e^{S^1}(N_X) = \left([\omega|_X] + t\big(\phi(Y)-\phi(X)\big)\right)^{\frac{1}{2}\dim(Y)+1} .$$
\end{lemma}

\begin{proof}
Since $[\omega]$ is an integral class, by Lemma~\ref{ut}, there exists $\ut\in H^2_{S^1}(M; \Z)$  such that
$$\ut|_X = [\omega|_X] \,\,\mbox{and}\,\, \ut|_Y =[\omega|_Y] + t \big(\phi(X)-\phi(Y)\big).$$ 
So 
\begin{equation}\label{ursy}
\left(\ut + t\big(\phi(Y)-\phi(X)\big)\right)^{\frac{1}{2}\dim(Y)+1}|_Y =0.
\end{equation}
Let $M^- = \{m\in M\,|\, \phi(X) <  \phi(m) \}.$
Consider the long exact sequence for the pair $(M, M^-)$ in equivariant cohomology:
\[\begin{array}{llclclcll}
\cdots & \to & H^*_{S^1}(M, M^-;\Z) & \to & H^*_{S^1}(M; \Z) & \to & H^*_{S^1}(M^-; \Z) & \to\cdots \\
   &      &   \downarrow \cong         &            &   \downarrow        &      &      \downarrow \cong  &      & \\
   &      &   H^{*-\codim(X)}_{S^1}(X;\Z) & \to     &  H^*_{S^1}(X;\Z)  & \to     &      H^*_{S^1}(Y; \Z)    &       &
\end{array}
\]
Here, the first vertical map is the Thom isomorphism, the second vertical map is the restriction, the third
vertical map is an isomorphism since $M^-$ is homotopy equivalent to $Y$, and the first horizontal map on the
second row is multiplication by $e^{S^1}(N_X)$. Combining (\ref{ursy}),  we get that there exists
$\lambda\in H^{\dim(Y)+2-\codim(X)}_{S^1}(X; \Z)$ such that
$$\lambda e^{S^1}(N_X) = \left(\ut + t\big(\phi(Y)-\phi(X)\big)\right)^{\frac{1}{2}\dim(Y)+1}|_X$$
$$=\left([\omega|_X] + t\big(\phi(Y)-\phi(X)\big)\right)^{\frac{1}{2}\dim(Y)+1}.$$
\end{proof}

Next we introduce the localization formula  due to Atiyah-Bott, and Berline-Vergne \cite{AB, BV}.

\begin{theorem}\label{AB.BV}
Let the circle act on a compact oriented manifold $M$. Fix a class $\alpha\in H^*_{S^1}(M; \Q)$. Then as elements of $\Q(t)$,
$$\int_M \alpha = \sum_{F\subset M^{S^1}}\int_F \frac{\alpha|_F}{e^{S^1}(N_F)},$$
where the sum is over all fixed point set components.
\end{theorem}

For the integral on $M$ or on $F$,  only the term containing the volume form of the corresponding manifold
contributes to the integral. 

In later sections, when we use  Theorem~\ref{AB.BV}, we will encounter a pure algebraic fact:
\begin{lemma}\label{n-1n}
Let $n\in\N$ and $n\geq 2$. Let $A_n$, $B_{n-1}$ and $C_{n-2}$ be respectively the coefficients of 
$w^n$, $w^{n-1}$ and $w^{n-2}$  in $(1+w+w^2+\cdots + w^n)^{n+1}$. Then
$A_n=2B_{n-1}\neq 0$ and $B_{n-1}\neq 2C_{n-2}\neq 0$.
\end{lemma}

\begin{proof}
Using induction on $k$, we can prove 
$$(1+w+\cdots + w^n +\cdots)^{k+1} = \sum_{n=0}^{\infty}{n+k\choose k} w^n.$$
So
$$A_n = {2n\choose n},\,\, B_{n-1}={2n-1\choose n},\,\, \mbox{and}\,\,\,C_{n-2}={2n-2\choose n}.$$
It is easy to check that for $n \geq 2$, $A_n=2B_{n-1}\neq 0$ and $B_{n-1} > 2C_{n-2}\neq 0$.
\end{proof}

\medskip

Next we state the relation between the total Chern class, the equivariant total  Chern class, and the equivariant Euler class of an $S^1$-(almost) complex vector bundle over a compact manifold. 

\begin{lemma}\label{chern}\cite[Lemma 2.4]{LT}
Let the circle act on a complex vector bundle $E$ of complex rank $d$ over a
compact manifold $X$ so that $E^{S^1} = X$. Assume that there exists
a non-zero $\lambda \in \Z$ so that the circle acts on $E$ with
weight $\lambda$. Then there exists $c_i \in H^{2i}(X;\Z)$ for all
$0\leq i \leq d$ such that
\begin{align*}
c(E) &= 1 + c_1 + \dots + c_{d-1} + c_d, \\
c^{S^1}(E) &= (1 + \lambda t)^d + c_1 (1 + \lambda t)^{d-1}+ \dots + c_{d-1} (1 + \lambda t) +c_d,
\,\, \mbox{and} \\
e^{S^1}(E) &=
(\lambda t)^d + c_1 (\lambda t)^{d-1}+ \dots +  c_{d-1} (\lambda t)
+c_d.
\end{align*}
Here, $c(E)$, $c^{S^1}(E)$, and $e^{S^1}(E)$ are the total Chern
class of $E$, the {\em equivariant} total  Chern class of $E$, and
the  equivariant Euler class of $E$, respectively.
\end{lemma}
When we work with $e^{S^1}(N_X)$ in later sections, we will need the following result. 
\begin{lemma}\label{basic}
Let $n>1$ and $n\in\N$, $m\in\N$, $a_0, a_1, \cdots, a_n\in\Z$,   $t$ and  $u$ be variables. If
$$\left(t+\frac{a_0}{m} u\right)(t^n + a_1 t^{n-1}u + \cdots +a_n u^n)=\left(t+\frac{u}{m}\right)^{n+1} \mod u^{n+1}                                   , \,\,\mbox{or} $$
$$t\,\left(t+\frac{a_0}{m} u\right)(t^{n-1} + a_1 t^{n-2}u + \cdots +a_{n-1} u^{n-1})=\left(t+\frac{u}{m}\right)^{n+1} \mod u^{n+1}$$
holds, then $m=1$.
\end{lemma}

\begin{proof}
Comparing the coefficients of $tu^n$ on both sides of  the equality, we get
$$m^{n-1}|(n+1).$$
So $m=1$ if $n\geq 4$. For $n=2$ and $n=3$, if $m\neq 1$, then $m=3$ and $m=2$ respectively, for these two
possibilities, comparing the coefficients containing $u$ and $u^2$ on both sides of the equality,  we see
that they are not possible. Hence $m=1$ for all $n\geq 2$.
\end{proof}

\section{proof of Theorem~\ref{c1}}\label{thm1}

In this section, we prove Theorem~\ref{c1}.

First, we prove two elementary lemmas.
\begin{lemma}\label{prim}
Let $(M, \omega)$ be a compact symplectic manifold admitting a Hamiltonian $S^1$ action such that $M^{S^1}$ consists of two connected components $X$ and $Y$. If $\codim(Y)>2$, then
$b_2(M)=b_2(X)$, and $[\omega]$ is primitive integral if and only if $[\omega|_X]$ is primitive integral.
Similarly, if $\codim(X)>2$, then  $b_2(M)=b_2(Y)$,  and $[\omega]$ is primitive integral if and only if  $[\omega|_Y]$ is primitive integral.
\end{lemma}

\begin{proof}
Let $\phi$ be the moment map and assume $\phi(X) < \phi(Y)$.
Using $\phi$ as a Morse-Bott function, $\codim(Y)$ is the Morse index of $Y$. If $\codim(Y) >2$, then the restriction map $H^2(M; \Z)\to H^2(X; \Z)$ is an isomorphism, so $b_2(M)=b_2(X)$, and 
$[\omega]$ is primitive integral if and only if $u=[\omega|_X]$  is primitive integral. Similarly, using $-\phi$, we get the other claims.
\end{proof}

\begin{lemma}\label{halfdim}
Let $(M, \omega)$ be a compact symplectic manifold admitting a Hamiltonian $S^1$ action such that $M^{S^1}$ consists of two connected components $X$ and $Y$ with $\dim(X)+\dim(Y)=\dim(M)$. 
If $\dim(X)\geq \frac{1}{2}\dim(M)$ and $b_{2i}(X)=1$ for all $0\leq 2i\leq\dim(X)$, then $\dim(X) = \frac{1}{2}\dim(M)$.
\end{lemma}

\begin{proof}
 Let $\phi$ be the moment map and assume $\phi(X)<\phi(Y)$. Since $\phi$ is a perfect
 Morse-Bott function, we have
\begin{equation}\label{perfect}
\dim H^i(M) = \dim H^i(X) + \dim H^{i-\codim(Y)}(Y), \,\,\forall \,\, i.
\end{equation}
In our case,  $\codim(Y)=\dim(X)=2k$ for some $k$. So $b_{2i}(X)=1$,  $\forall \,\, 0\leq 2i\leq 2k$ implies that 
\begin{equation}\label{dual}
b_{2i}(M) =1,\,\, \forall \,\, 0\leq 2i < 2k, \,\,\mbox{and}\,\,\, b_{2k}(M)=2.
\end{equation}
 If $\dim(X) > \frac{1}{2}\dim(M)$, then 
$\dim(M)-2k < 2k$. Since $M$ is compact and oriented, by
Poinc\'are duality, $b_{\dim(M)-2k}(M)=b_{2k}(M)=2$, which contradicts  (\ref{dual}). 
\end{proof}

The next result, Proposition~\ref{euler},  is a main step toward proving Theorem~\ref{c1}. 
It is also an important step for determining $e^{S^1}(N_X)$.

\begin{proposition}\label{euler}
Let $(M, \omega)$ be a compact symplectic manifold of dimension bigger than $4$ admitting a Hamiltonian $S^1$ action with moment map $\phi$ such that $M^{S^1}$ consists of two connected components $X$ and $Y$ with
$\dim(X)+\dim(Y)=\dim(M)$ and $\phi(X) < \phi(Y)$. Assume $[\omega]$ is a primitive integral class, $2n=\dim(X)\geq\frac{1}{2}\dim(M)$, and $H^{\mbox{even}}(X; \Z)=\Z[u]/u^{n+1}$, where $u=[\omega|_X]$.
If  the action is semifree,  then
$$\phi(Y)-\phi(X)=1,$$
and there exists $a_0\in\Z$ such that
$$(t+ a_0 u) e^{S^1}(N_X) = (t+u)^{n+1}.$$
 \end{proposition}

\begin{proof}
By the assumptions and Lemma~\ref{halfdim}, $2n=\dim(X)=\frac{1}{2}\dim(M)=\dim(Y)>2$. Since
$\codim(Y)=\dim(X)>2$, by Lemma~\ref{prim}, $u=[\omega|_X]$  is primitive integral.

Since $[\omega]$ is an integral class,  $\phi(Y)-\phi(X)=m \in \N$.

By Lemma~\ref{mult}, there exists $\lambda\in H^*_{S^1}(X; \Z)$ such that
\begin{equation}\label{lam}
\lambda e^{S^1}(N_X) = \left (mt+u\right)^{n+1}.
\end{equation}
Since the action is semifree, and $\rank_{\C}(N_X) =\frac{1}{2}\dim(Y)=n$, by Lemma~\ref{chern}, and 
the assumption $H^{\mbox{even}}(X; \Z)=\Z[u]/u^{n+1}$, we have
$$e^{S^1}(N_X)  = t^n + a_1 t^{n-1}u + \cdots + a_n u^n, \,\,\mbox{where $a_i \in \Z, \,\forall \,i$}.$$
By degree reasons and by comparing the coefficients of $t^{n+1}$ on both sides of (\ref{lam}),  we may let
$$\lambda =m^{n+1} t+au, \,\mbox{with $a\in\Z$}.$$
Comparing the coefficients of $t^nu$ on both sides of (\ref{lam}),  
we get
$$a=a_0 m^n\,\,\mbox{for some $a_0\in\Z$}.$$ 
So we may write (\ref{lam})  as
\begin{equation}\label{lam2}
\big(t+\frac{a_0}{m}u\big)(t^n + a_1 t^{n-1}u + \cdots + a_n u^n) = \left (t + \frac{u}{m}\right)^{n+1} \mod u^{n+1}.
\end{equation}
 By Lemma~\ref{basic}, $m=1$. Both claims follow.
\end{proof}

To prove Theorem~\ref{c1}, let us also recall the following results.

\begin{lemma}\label{c1L}\cite[Lemma 2.3]{L}
Let the circle act on a connected compact symplectic manifold
$(M, \omega)$ with moment map $\phi\colon M\to\R$. Assume $b_2(M)=1$.
 Then
$$c_1(M) = \frac{ \Gamma_{F} - \Gamma_{F'}}{\phi(F') - \phi(F)} [\omega],$$
where $F$ and $F'$ are any two fixed components such that  $\phi(F') \neq \phi(F)$.
\end{lemma}

\begin{proposition}\label{<6}\cite[Proposition 7.5]{LT}
Let $(M, \omega)$ be a compact symplectic manifold admitting an effective Hamiltonian $S^1$ action with moment map $\phi$ such that $M^{S^1}$ consists of two connected components $X$ and $Y$ with $\phi(X)<\phi(Y)$. Then
the set of distinct weights of the $S^1$ action on the normal bundles $N_X$ of $X$ and $N_Y$ of $Y$ are respectively  $\{1, 2, \cdots, N\}$ and $\{-1, -2, \cdots, -N\}$ for some $N\in\N$.
\end{proposition}

Now we are ready to prove Theorem~\ref{c1}.

\begin{proof}[Proof of Theorem~\ref{c1}]
 With no loss of generality, assume $\phi(X) < \phi(Y)$.
 
By Lemma~\ref{halfdim}, $\dim(X)=\frac{1}{2}\dim(M)$. The fact $\codim(Y)>2$ and Lemma~\ref{prim} imply
that $b_2(M)=b_2(X)=1$.

By Proposition~\ref{euler}, $(1)\Longrightarrow (2)$. Conversely, if there exists any nontrivial finite stabilizer, then by
Lemma~\ref{ut}, $\phi(Y)-\phi(X) > 1$. Hence $(2)\Longrightarrow (1)$.

If the action is semifree, then $\Gamma_X=\rank_{\C}(N_X)= \frac{1}{2}\dim(Y)$,
and similarly $\Gamma_Y= - \frac{1}{2}\dim(X)$. Moreover, by the last step, $\phi(Y)-\phi(X) = 1$.  Then $(3)$ follows from Lemma~\ref{c1L}. This shows $(1)\Longrightarrow (3)$.

To prove $(3)\Longrightarrow (1)$, we only need $b_2(X)=b_2(M)=1$. Assume the action is not semifree. By Propositions~\ref{1+2} and \ref{<6}, $\dim(X)=\dim(Y)$, and the set of distinct weights on $N_X$ is $\{1, 2, \cdots, N\}$ for some $N > 1$. Let $m_i$ be the multiplicity of the weight $i$ on $N_X$. Then 
$\sum_i m_i =\rank_{\C}(N_X)=\frac{1}{2}\dim(Y)$, and
$$\Gamma_X = Nm_{N} +(N-1)m_{N-1}+\cdots 1\cdot m_1 < N \frac{1}{2}\dim(Y),\,\, \mbox{and}\,\,\,  \Gamma_Y = -\Gamma_X.$$ 
By Lemma~\ref{ut},  $\phi(Y)-\phi(X)\geq N\cdot (N-1)$. By Lemma~\ref{c1L}, 
$$c_1(M)=\frac{\Gamma_X -\Gamma_Y}{\phi(Y)-\phi(X)}[\omega]=\frac{2\Gamma_X}{\phi(Y)-\phi(X)}[\omega],$$
and by the information above, 
$$\frac{2\Gamma_X}{\phi(Y)-\phi(X)} < \dim(Y) = \frac{1}{2}\dim(M),$$
contradicting to $c_1(M)=\frac{1}{2}\dim(M)[\omega]$.
\end{proof}

\section{determining the action}\label{action}

In this section, under a cohomology condition on both fixed point set components,
we prove that non-semifree actions do not exist. 
\smallskip

First, let us look at what stabilizer groups can occur.
\begin{proposition}\label{1+2}\cite[Lemma 7.1 and Proposition 7.9]{LT}
Let $(M, \omega)$ be a compact symplectic manifold admitting an effective Hamiltonian $S^1$ action with $M^{S^1}$ consisting of two connected components $X$ and $Y$. Suppose that the action is not semifree. Then
\begin{enumerate}
\item $\dim(X)=\dim(Y)$, and
\item if additionally $b_{2i}(X)=1$ for all $0\leq 2i\leq\dim(X)$, 
then the only finite stabilizer groups are $1$ and $\Z_2$, and $\dim(M^{\Z_2})-\dim(X)=2$ or $\dim(M)-\dim(M^{\Z_2}) =2$ or both.
\end{enumerate}
\end{proposition}

The next fact was implied by \cite[Lemma 7.6]{LT}, but was stated differently. We state it as follows, and will use it
twice.
 
\begin{lemma}\label{bothprim}
Let $(M, \omega)$ be a compact symplectic manifold admitting a Hamiltonian $S^1$ action such that $M^{S^1}$ consists of two connected components $X$ and $Y$ such that $\dim(M)-\dim(X)=\dim(M)-\dim(Y)=2$. Assume $b_2(X)=b_2(Y)=1$,  $[\omega|_X]$ and $[\omega|_Y]$ are both primitive integral.
Then $c_1(N_X)=c_1(N_Y)=0$.
\end{lemma}

Next, in Lemmas~\ref{12} and \ref{4t}, assuming a cohomology condition on one fixed set component,  
we determine the equivariant Euler class of its normal bundle.

\begin{lemma}\label{12}
Let $(M, \omega)$ be a compact symplectic manifold  admitting an effective non-semifree Hamiltonian $S^1$ action with moment map $\phi$ such that $M^{S^1}$ consists of two connected components $X$ and $Y$
with $\dim(X)+\dim(Y)=\dim(M)$ and $\phi(X) <\phi(Y)$. Assume $[\omega]$ is a primitive integral class, and
$H^{\mbox{even}}(X; \Z)=\Z[u]/u^{n+1}$, where $u=[\omega|_X]$  and $2n=\dim(X)$.  
Then $\dim(X)=\dim(Y)=2n\geq 6$,  the only finite stabilizer groups are $1$ and $\Z_2$, and 
\begin{equation}\label{a0}
  \dim(M^{\Z_2})-\dim(X) > 2,\,\,  \dim(M)-\dim(M^{\Z_2}) =2,
\end{equation}
\begin{equation}\label{b0}
\phi(Y)-\phi(X) = 2,\,\,\mbox{and}\,\,\, e^{S^1}\big(N_{M^{\Z_2}}\big)|_X = t+u.
\end{equation}
\end{lemma}

\begin{proof}
Since the action is not semifree, by Proposition~\ref{1+2}, $\dim(X)=\dim(Y)=\codim(Y)=2n$ for some $n\geq 2$. 
By Lemma~\ref{prim}, $b_2(X)=b_2(M)=b_2(Y)=1$,
$u=[\omega|_X]$ and $v=[\omega|_Y]$ are  primitive integral. After (\ref{a0}) is shown, we get $2n\geq 6$.

Since $b_{2i}(X) =1$ for all $0\leq 2i\leq 2n$, by Proposition~\ref{1+2}, the only finite stabilizer groups are $1$ and $\Z_2$. Since $[\omega]$ is integral, $m=\phi(Y)-\phi(X)\in\N$. Since there is $\Z_2$ stabilizer, by Lemma~\ref{ut},
\begin{equation}\label{2|}
  2\,|\, m.
\end{equation}

 Assume instead $\dim(M^{\Z_2})-\dim(X)=2$. Then by Lemmas~\ref{bothprim} and \ref{chern},
 $$e^{S^1}(N_X^{M^{\Z_2}}) = 2t,$$ 
where $N_X^{M^{\Z_2}}$ is the normal bundle of $X$ in $M^{\Z_2}$.
The action on $N_{M^{\Z_2}}|_X$ is semifree, and  $\rank_{\C}\big(N_{M^{\Z_2}}|_X\big) = n-1$. 
Since $H^{\mbox{even}}(X; \Z)=\Z[u]/u^{n+1}$, by Lemma~\ref{chern}, we may write
$$e^{S^1}(N_{M^{\Z_2}})|_X = t^{n-1} + a_1t^{n-2}u + \cdots + a_{n-1}u^{n-1},\,\,\mbox{with $a_i\in\Z,\,\forall\, i$}.$$
We have
\begin{equation}\label{2normal}
e^{S^1}(N_X)=e^{S^1}(N_X^{M^{\Z_2}})e^{S^1}(N_{M^{\Z_2}})|_X.
\end{equation}
By Lemma~\ref{mult}, there exist  $a, b\in\Z$ such that
$$(at + bu)\cdot 2t \cdot (t^{n-1} + a_1t^{n-2}u + \cdots + a_{n-1}u^{n-1}) = (mt+u)^{n+1}.$$
Comparing the coefficients of $t^{n+1}$ and $t^nu$ on both sides, we get that there exists $a_0\in\Z$ such that
$$t\,\left(t + \frac{a_0}{m} u\right) (t^{n-1} + a_1t^{n-2}u + \cdots + a_{n-1}u^{n-1}) = \left(t+\frac{u}{m}\right)^{n+1}\mod u^{n+1}.$$
By Lemma~\ref{basic}, $m=1$, which contradicts  (\ref{2|}). 
Together with Proposition~\ref{1+2}, (\ref{a0}) follows.

Since  $\dim(M^{\Z_2})-\dim(X) > 2$ and $b_2(X)=1$, by \cite[Lemma 7.7]{LT},  
\begin{equation}\label{nl}
c_1\big(N_{M^{\Z_2}}\big)|_X = 2 \frac{\Gamma_1}{m}u=\frac{2}{m}u,
\end{equation}
where $\Gamma_1$ is the sum of the weights $1$'s on the normal bundle to $X$, which is $1$ here.
Since $\frac{2}{m}$ needs to be an integer, we have $m\,|\,2$. Together with (\ref{2|}), 
(\ref{nl}), and that $N_{M^{\Z_2}}|_X$ is a complex line bundle and the weight of the action on it is $1$, 
(\ref{b0}) follows. 
\end{proof}

\begin{lemma}\label{4t}
Assume the assumptions of Lemma~\ref{12} hold. Then $\dim(X)=\dim(Y)=2n$
with $n\geq 3$ being odd, and 
$$4t \,e^{S^1}(N_X)=(2t+u)^{n+1}.$$
\end{lemma}

\begin{proof}
By Lemma~\ref{12}, $\dim(X)=\dim(Y)=2n\geq 6$, and
 $$\phi(Y)-\phi(X) = 2, \,\,\mbox{and}\,\,\, e^{S^1}\big(N_{M^{\Z_2}}\big)|_X = t + u.$$
By (\ref{a0}) and Lemma~\ref{chern}, we may write
$$e^{S^1}(N_X^{M^{\Z_2}}) = (2t)^{n-1}+a_1(2t)^{n-2}u+ \cdots +a_{n-1}u^{n-1},\,\,\mbox{with $a_i\in\Z, \,\forall\, i$}.$$  
By (\ref{2normal}) and Lemma~\ref{mult},  there exist $c, d\in\Z$ such that
$$(ct+du)(t+u)\big((2t)^{n-1}+a_1(2t)^{n-2}u+\cdots +a_{n-1}u^{n-1}\big) = (2t+u)^{n+1}.$$
Comparing the coefficients of $t^{n+1}$ and of $t^nu$ on both sides, we get 
$c=4$,  and $d=2a_0$ for some $a_0\in\Z$.
Hence
\begin{equation}\label{prod}
(2t+a_0u)(2t+2u)\big((2t)^{n-1}+a_1(2t)^{n-2}u+\cdots+a_{n-1}u^{n-1}\big) 
\end{equation}
$$= (2t+u)^{n+1}  \mod u^{n+1}.$$ 
So there exists $\lambda\in\C$ such that 
$(2t+u)^{n+1} + (\lambda u)^{n+1}$, as a polynomial, is equal to the left hand side of (\ref{prod}). We have 
$1+\lambda^{n+1}=2a_0a_{n-1}$ (by comparing the coefficients of $u^{n+1}$ on both sides), and 
$$(2t+u)^{n+1} + (\lambda u)^{n+1} = \prod_{k=0}^{n} \left(2t+u + e^{\frac{2\pi k \sqrt{-1}}{n+1}}\lambda u\right).$$ 
Since there is a linear factor $2t+2u$ on the left hand side of (\ref{prod}), there exists a $k$ such that $e^{\frac{2\pi k \sqrt{-1}}{n+1}}\lambda u =u$. So $|\lambda| =1$. Since there is another linear factor $2t+a_0u$ on the left hand side of (\ref{prod}), there is another $k'$ such that $e^{\frac{2\pi k' \sqrt{-1}}{n+1}}\lambda u$ is real and it must be $-u$. Hence
$n$ must be odd, and the linear factor $2t+a_0u = 2t$. 
 \end{proof}

Now, we reach our final result of this section: 

\begin{proposition}\label{nex}
There exists no compact symplectic manifold $(M, \omega)$ admitting an effective non-semifree
 Hamiltonian $S^1$ action such that $M^{S^1}$ consists of two connected components $X$ and $Y$
satisfying $\dim(X)+\dim(Y)=\dim(M)$, $H^{\mbox{even}}(X; \Z) =\Z[u]/u^{\frac{1}{2}\dim(X)+1}$ and 
$H^{\mbox{even}}(Y; \Z) =\Z[v]/v^{\frac{1}{2}\dim(Y)+1}$, where $u=[\omega|_X]$, $v=[\omega|_Y]$ and
$[\omega]$ is integral.
\end{proposition}

\begin{proof}
Assume such a symplectic manifold $(M, \omega)$ exists, the moment map is $\phi$ and $\phi(X)<\phi(Y)$.
By Lemma~\ref{12}, $\dim(X)=\dim(Y)=2n$ for some $n\geq 3$. 
We may assume $[\omega]$ is primitive.
By Lemma~\ref{4t},
$$e^{S^1}(N_X)=\frac{(2t+u)^{n+1}}{4t}.$$
Similarly, by symmetry,
$$e^{S^1}(N_Y)=\frac{(-2t+v)^{n+1}}{-4t}.$$
Using Theorem~\ref{AB.BV} to integrate $1$ on $M$, we get
$$0=\int_X \frac{1}{e^{S^1}(N_X)} + \int_Y \frac{1}{e^{S^1}(N_Y)},$$
from which, we get
\begin{equation}\label{int1}
0=\int_X \frac{1}{(2t+u)^{n+1}}+ (-1)^n \int_Y \frac{1}{(2t-v)^{n+1}}.
\end{equation}
Let $w = -u$, then $w^n = (-1)^n u^n$, so $\int_X w^n = (-1)^n$. From (\ref{int1}), we get 
\begin{equation}\label{int2}
0=\int_X \frac{1}{\big(1-\frac{w}{2t}\big)^{n+1}} + (-1)^n \int_Y \frac{1}{\big(1-\frac{v}{2t}\big)^{n+1}}.
\end{equation}
We have
$$\frac{1}{\big(1-\frac{w}{2t}\big)^{n+1}} = \Big(1+ \frac{w}{2t} + \cdots + \big(\frac{w}{2t}\big)^n \Big)^{n+1}.$$
Let $A$ be the coefficient of  $\big(\frac{w}{2t}\big)^n$  in this expression. Then $A > 0$, and
$A$ is also the coefficient of $\big(\frac{v}{2t}\big)^n$  in $\frac{1}{\big(1-\frac{v}{2t}\big)^{n+1}}$. Then (\ref{int2}) gives
$$0 = (-1)^n A + (-1)^n A, \,\,\mbox{a contradiction}.$$
\end{proof}

\section{The equivariant Euler classes of the normal bundles of the fixed point sets}\label{sec-euler}

In this section, assuming a cohomology condition on both fixed point set components, we determine the equivariant Euler classes of the normal bundles of the fixed point set components.

\begin{proposition}\label{exy}
Let $(M, \omega)$ be a compact symplectic manifold  admitting an effective Hamiltonian $S^1$ action with moment map $\phi$ such that $M^{S^1}$ consists of two connected components $X$ and $Y$ with $\dim(X)+\dim(Y)=\dim(M)$
and $\phi(X) < \phi(Y)$. Assume $[\omega]$ is primitive integral, 
$$H^{\mbox{even}}(X; \Z)=\Z[u]/u^{\frac{1}{2}\dim(X)+1}\,\,\mbox{and}\,\,\, H^{\mbox{even}}(Y; \Z)=\Z[v]/v^{\frac{1}{2}\dim(Y)+1},$$ 
where $u=[\omega|_X]$  and $v=[\omega|_Y]$. Then  the action must be semifree,
$\dim(X)=\dim(Y)=2n$ with $n\geq 1$,  
\begin{equation}\label{equi-euler}
e^{S^1}(N_X) =\frac{(t+u)^{n+1}}{t+2u}, \,\,\mbox{and}\,\,\,  e^{S^1}(N_Y) =\frac{(-t+v)^{n+1}}{-t+2v}.
\end{equation}
\end{proposition}

\begin{proof}
By Proposition~\ref{nex}, the action must be semifree.

Since $\dim(X)+\dim(Y)=\dim(M)$, we have $\dim(X)\geq\frac{1}{2}\dim(M)$ or $\dim(Y)\geq\frac{1}{2}\dim(M)$. By Lemma~\ref{halfdim}, we get $\dim(X)=\dim(Y)=2n=\frac{1}{2}\dim(M)$ for some $n\geq 1$.

 First, assume $n=1$, i.e., $\dim(X)=\dim(Y)=2$ and $\dim(Y)=4$. In this case 
$\rank_{\C}(N_X)=\rank_{\C}(N_Y)=1$. The assumption $H^2(X; \Z)=\Z[u]/u^2$ and $H^2(Y; \Z)=\Z[v]/v^2$ means that $u$ and $v$ are both primitive integral. By Lemma~\ref{bothprim}, 
$$e^{S^1}(N_X) = t +0u, \,\,\mbox{and}\,\,\,e^{S^1}(N_Y) = -t +0 v.$$
Hence (\ref{equi-euler}) holds for dimension $4$. 

Next, assume $n > 1$.  By Proposition~\ref{euler}, 
$$e^{S^1}(N_X) = \frac{\left (t+u\right)^{n+1}}{t+ au},\,\,\mbox{with $a\in\Z$}.$$
Similarly, by symmetry,
$$e^{S^1}(N_Y) = \frac{\left (-t+v\right)^{n+1}}{-t+ bv},\,\,\mbox{with $b\in\Z$}.$$
Using Theorem~\ref{AB.BV} to integrate $1$ on $M$, we get
$$\int_X \frac{1}{e^{S^1}(N_X)} + \int_Y \frac{1}{e^{S^1}(N_Y)} =0,$$
from which we get
$$\int_X \frac{t+au}{(t+u)^{n+1}} + (-1)^n\int_Y \frac{t-bv}{(t-v)^{n+1}} =0.$$
Let $w=-u$, then $w^n=(-1)^n u^n$, and $\int_X w^n = (-1)^n$. The above integral becomes
$$\int_X \big(t-aw\big)\Big(1+\frac{w}{t}+\cdots+\big(\frac{w}{t}\big)^n\Big)^{n+1} + (-1)^n\int_Y \big(t-bv\big)\Big(1+\frac{v}{t}+\cdots+\big(\frac{v}{t}\big)^n\Big)^{n+1}=0.$$  
Let $A_n$ and $B_{n-1}$ be respectively the coefficients of $\big(\frac{w}{t}\big)^n$ and $\big(\frac{w}{t}\big)^{n-1}$
in the expression $\Big(1+\frac{w}{t}+\cdots + (\frac{w}{t})^n\Big)^{n+1}$. Then the above integral gives
$$(-1)^n (A_n -aB_{n-1}) + (-1)^n (A_n - bB_{n-1})=0.$$
By Lemma~\ref{n-1n}, $A_n=2B_{n-1}\neq 0$. So
\begin{equation}\label{a+b}
a+b =4.
\end{equation}
Next, we integrate $c_1^{S^1}(M)$ on $M$.  By Theorem~\ref{c1}, $c_1(M)|_X=2nu$ and $c_1(M)|_Y=2nv$.
Moreover, $\Gamma_X=n$ and $\Gamma_Y=-n$. So we have
$$c_1^{S^1}(M)|_X = nt + 2nu, \,\,\mbox{and}\,\, c_1^{S^1}(M)|_Y = -nt + 2nv.$$
Since $\dim(M)>2$,  we have
$$\int_X \frac{c_1^{S^1}(M)|_X}{e^{S^1}(N_X)} + \int_Y \frac{c_1^{S^1}(M)|_Y}{e^{S^1}(N_Y)} =0.$$
Similar to the above, let $w=-u$, we get
$$\int_X\frac{(nt - 2nw)(t-aw)}{(t-w)^{n+1}} + (-1)^{n-1}\int_Y\frac{(nt - 2nv)(t-bv)}{(t-v)^{n+1}}=0,$$
from which we get
$$\int_X \Big(nt^2-(2n+na)tw+2naw^2\Big)\Big(1+\frac{w}{t}+\cdots+\big(\frac{w}{t}\big)^n\Big)^{n+1}$$
$$+(-1)^{n-1}\int_Y \Big(nt^2-(2n+nb)tv+2nbv^2\Big)\Big(1+\frac{v}{t}+\cdots+\big(\frac{v}{t}\big)^n\Big)^{n+1}=0.$$
Let $C_{n-2}$ be the coefficient of $\big(\frac{w}{t}\big)^{n-2}$ in $\Big(1+\frac{w}{t}+\cdots + (\frac{w}{t})^n\Big)^{n+1}$.
Then the integral gives
$$(-1)^n\Big(nA_n -(2n+na)B_{n-1} + 2naC_{n-2}\Big) $$
$$+ (-1)^{n-1}\Big(nA_n -(2n+nb)B_{n-1} + 2nbC_{n-2}\Big)=0.$$
This simplifies to
$$(a-b)(B_{n-1} - 2C_{n-2})=0.$$
By Lemma~\ref{n-1n},  $B_{n-1} \neq 2C_{n-2}\neq 0$. So
\begin{equation}\label{a=b}
 a=b.
\end{equation}
By (\ref{a+b}) and (\ref{a=b}),  $a=b=2$,  so (\ref{equi-euler}) follows.
\end{proof}

\section{the integral cohomology rings of the fixed point sets}\label{sec-ring}
In this section, for $\dim(M)>4$, with the cohomology condition in even degrees on both fixed set components, 
we find the ring $H^*(X; \Z)$ and $H^*(Y; \Z)$.

\begin{lemma}\label{rxy}
Let $(M, \omega)$ be a compact symplectic manifold of dimension bigger than $4$ admitting a Hamiltonian $S^1$ action such that $M^{S^1}$ consists of two connected components $X$ and $Y$ with $\dim(X)+\dim(Y)=\dim(M)$. Assume $[\omega]$ is primitive integral, 
$$H^{\mbox{even}}(X; \Z)=\Z[u]/u^{\frac{1}{2}\dim(X)+1}\,\,\,\mbox{and}\,\,\, H^{\mbox{even}}(Y; \Z)=\Z[v]/v^{\frac{1}{2}\dim(Y)+1},$$ 
where $u=[\omega|_X]$  and $v=[\omega|_Y]$. Then $\dim(X)=\dim(Y)=2n > 2$, 
$$H^*(X; \Z)= \Z[u]/u^{n+1}\, \,\,\mbox{and}\,\,\,  H^*(Y; \Z)=\Z[v]/v^{n+1}.$$
\end{lemma}

\begin{proof}
Let $\phi$ be the moment map and assume $\phi(X) < \phi(Y)$. By quotienting out a finite subgroup action if necessary, we may assume the action is effective.
By Proposition~\ref{exy}, the action is semifree, $\dim(X)=\dim(Y)=2n >2$  (we assumed $\dim(M) > 4$), and
$$e^{S^1}(N_Y) = (-t)^n + (n-1)v (-t)^{n-1} + \lot $$
$$= (-1)^n \big( t^n + (1-n)v t^{n-1} + \lot \big).$$
By \cite[Lemma 8.5]{LT}, if  there exists a class $\ut$ such that  $\ut|_x =0$,  $\forall\, x\in X$ and
$\ut|_y \neq 0$,  $\forall\, y\in Y$, and there exists a class $\Tilde \mu$ such that $\Tilde \mu|_Y = (1-n)v$ (the degree
$2$ term, the coefficient of $t^{n-1}$ in $e^{S^1}(N_Y)$ above) and $\Tilde \mu|_x \neq -nt$,  $\forall \,x\in X$, then
$$H^{2k+1}(X; \R)=0\,\,\,\mbox{and}\,\,\, H^{2k+1}(X; \Z_p)=0$$ for all $k$ and all prime numbers $p$.
By Lemma~\ref{ut},  the required $\ut$ exists. Note that by Theorem~\ref{c1},  $\phi(Y)-\phi(X)=1$. We may take 
$$\Tilde \mu = (1-n)\ut +(1-n)\big(\phi(Y)-\phi(X)\big)t=(1-n)\ut +(1-n)t.$$
By the universal coefficient theorem, $H^*(X; \Z)$ has no torsion, and no odd degree terms. 
So $H^*(X; \Z) = \Z[u]/u^{n+1}$ holds.
Similarly, the claim holds for $Y$.
\end{proof}

\section{the total Chern classes of  the fixed point sets and of their normal bundles}\label{sec-totalchern}
In this section, we compute $c(N_X)$, $c(N_Y)$, $c(X)$ and $c(Y)$.
\smallskip

First,  Proposition~\ref{exy} and Lemma~\ref{chern} give us $c(N_X)$ and $c(N_Y)$.
\begin{lemma}\label{nxy}
Let $(M, \omega)$ be a compact symplectic manifold  admitting a Hamiltonian $S^1$ action 
such that $M^{S^1}$ consists of two connected components $X$ and $Y$ with $\dim(X)+\dim(Y)=\dim(M)$. 
Assume $[\omega]$ is primitive integral, $$H^{\mbox{even}}(X; \Z)=\Z[u]/u^{\frac{1}{2}\dim(X)+1}\,\,\,\mbox{and}\,\,\, H^{\mbox{even}}(Y; \Z)=\Z[v]/v^{\frac{1}{2}\dim(Y)+1},$$ 
where $u=[\omega|_X]$  and $v=[\omega|_Y]$. Then $\dim(X)=\dim(Y)=2n$ with $n\geq 1$, 
$$c(N_X) =\frac{(1+u)^{n+1}}{1+2u}, \,\,\mbox{and}\,\,\,  c(N_Y) =\frac{(1+v)^{n+1}}{1+2v}.$$
\end{lemma}

To get $c(X)$ and $c(Y)$, we will use Proposition~\ref{exy}, Lemma~\ref{rxy}  and the following result.
\begin{proposition}\label{isom}\cite[Prop. 5.1 and Cor. 5.2]{LT}
Let $(M, \omega)$ be a compact symplectic manifold admitting a Hamiltonian $S^1$ action with moment map $\phi$
such that $M^{S^1}$ consists of two connected components $X$ and $Y$. Assume the action is semifree
or $H^*(M^{S^1}; \Z)$ has no torsion. Then there is an isomorphism
$$f\colon  H^*_{S^1}(X; \Z)/e^{S^1}(N_X)\to   H^*_{S^1}(Y; \Z)/e^{S^1}(N_Y)\,\,\mbox{such that}$$
$$f(\Tilde\alpha|_X) = \Tilde\alpha|_Y, \,\, \, \forall \,\Tilde\alpha\in H^*_{S^1}(M; \Z).$$
In particular, if $[\omega]$ is integral, then
$$f([\omega|_X]) = [\omega|_Y] + t\left(\phi(X)-\phi(Y)\right).$$
\end{proposition}

\begin{proposition}\label{whole}
Let $(M, \omega)$ be a compact symplectic manifold of dimension bigger than $4$ admitting a Hamiltonian $S^1$ action such that $M^{S^1}$ consists of two connected components $X$ and $Y$ with $\dim(X)+\dim(Y)=\dim(M)$. Assume $[\omega]$ is primitive integral, 
$$H^{\mbox{even}}(X; \Z)=\Z[u]/u^{\frac{1}{2}\dim(X)+1}\,\,\,\mbox{and}\,\,\, H^{\mbox{even}}(Y; \Z)=\Z[v]/v^{\frac{1}{2}\dim(Y)+1},$$ 
where $u=[\omega|_X]$  and $v=[\omega|_Y]$. Then 
$\dim(X)=\dim(Y)=2n > 2$, 
\begin{equation}\label{3}
c(X)=(1+u)^{n+1},\,\,\mbox{and} \,\,\, c(Y)=(1+v)^{n+1}.
\end{equation}
\end{proposition}

\begin{proof}
Let $\phi$ be the moment map, and assume $\phi(X)<\phi(Y)$. 

We may assume the action is effective. By Proposition~\ref{exy}, the action is semifree, $\dim(X)=\dim(Y)=2n > 2$,
and together with Lemma~\ref{chern}, we have
\begin{equation}\label{cnxy}
c^{S^1}(N_X) = \frac{\left (1+ t+u\right)^{n+1}}{1+ t+ 2u},\,\,\mbox{and}\,\,c^{S^1}(N_Y) = \frac{\left (1- t+v\right)^{n+1}}{1- t + 2v}.
\end{equation}
By Theorem~\ref{c1}, 
\begin{equation}\label{=1}
\phi(Y)-\phi(X) = 1.
\end{equation}
Since $H^{\mbox{even}}(X; \Z)=\Z[u]/u^{n+1}$, we may assume 
\begin{equation}\label{cx}
c(X) =1+ a_1 u + \cdots + a_iu^i + \cdots + a_n u^n,\,\,\mbox{with $a_i\in\Z$  for $1\leq i\leq n$}.
\end{equation}
 Moreover, by Lemma~\ref{rxy}, $H^{\mbox{odd}}(X; \R)=0$. So by the Euler characteristic formula, we first get
$$a_n = \int_X a_n u^n = \sum_{i=0}^n (-1)^i\dim H^i(X; \R) = n+1.$$
Next, we compute the other $a_i$'s. We have that 
\begin{equation}\label{cm}
c^{S^1}(M)|_X = c^{S^1}(N_X) c(X),\,\,\mbox{and}\,\, c^{S^1}(M)|_Y = c^{S^1}(N_Y) c(Y),
\end{equation}
and we may similarly write 
\begin{equation}\label{cy}
c(Y)=1 +\sum_{i=1}^n b_i v^i\,\, \mbox{with $b_i\in\Z$}.
\end{equation}
Consider  the map $f$ in Proposition~\ref{isom} composed with the restriction map
$H^*_{S^1}(Y; \Z)/e^{S^1}(N_Y)\to H^*_{S^1}(y; \Z)/(-t)^n$, where $y\in Y$ is a point, 
$$g\colon H^*_{S^1}(X; \Z)/e^{S^1}(N_X)\to H^*_{S^1}(Y; \Z)/e^{S^1}(N_Y) \to H^*_{S^1}(y; \Z)/(-t)^n.$$
First, by (\ref{=1}), we have $$g(u)=\big(v + (\phi(X)-\phi(Y)) t\big)|_y = -t.$$ 
Together with (\ref{cnxy}), (\ref{cx}),  (\ref{cm}), and (\ref{cy}),  we obtain
$$g(c^{S^1}(M)|_X) = \frac{1+ \sum_{i=1}^n a_i (-t)^i}{1-t} =\big(c^{S^1}(M)|_Y\big)|_y  = (1-t)^n \mod t^n,$$
which gives
$$1+\sum_{i=1}^n a_i (-t)^i = (1-t)^{n+1} \mod t^n.$$ 
From this formula, we can get $a_i$ for $1\leq i\leq n-1$. Together with the value of $a_n$ and (\ref{cx}),
we get (\ref{3}) for $c(X)$. The claim for $c(Y)$ follows similarly.
\end{proof}

\section{the integral cohomology ring and total Chern class of $M$ and the proof of Theorem~\ref{lg}}\label{sec-thm2}

In this section, we first obtain  $H^*_{S^1}(M; \Z)$, $H^*(M; \Z)$,  $c^{S^1}(M)$ and $c(M)$, 
then we prove Theorem~\ref{lg}.

\begin{proposition}\label{equibase}
Let $(M, \omega)$ be a compact symplectic manifold of dimension bigger than $4$ admitting an effective
Hamiltonian $S^1$ action with  moment map $\phi$ such that $M^{S^1}$ consists of two connected components $X$ and $Y$ with $\dim(X)+\dim(Y)=\dim(M)$. Assume $[\omega]$ is primitive integral, 
$$H^{\mbox{even}}(X; \Z)=\Z[u]/u^{\frac{1}{2}\dim(X)+1}\,\,\,\mbox{and}\,\,\, H^{\mbox{even}}(Y; \Z)=\Z[v]/v^{\frac{1}{2}\dim(Y)+1},$$ where $u=[\omega|_X]$  and $v=[\omega|_Y]$. Then $\dim(X)=\dim(Y)=2n >2$, 
 $$c^{S^1}(M) = \frac{(1+\ut)^{n+1}(1+t+\ut)^{n+1}}{ 1+t+2\ut}, \,\,\mbox{and}\,\,\, c(M) = \frac{(1+[\omega])^{2n+2}}{1+2[\omega]}.$$
Moreover,  a basis of $H^*_{S^1}(M; \Z)$ as an $H^*(\CP^{\infty}; \Z)$-module is:
$$\ut^i\in H^{2i}_{S^1}(M; \Z), \,\,\mbox{and} \,\,\, \frac{\ut^{n+1} (t+\ut)^i}{t+2\ut}\in H^{2n+2i}_{S^1}(M; \Z),\,\,\mbox{where $0\leq i\leq n$},$$
and
\[ H^*(M; \Z) = \left\{\begin{array}{ll}
 \Z[x, y]/(x^{n+1}-2xy, y^2),  & \,\,\mbox{if $n$ is odd},\\
\Z[x, y]/(x^{n+1}-2xy, y^2-x^ny), & \,\,\mbox{if $n$ is even},
\end{array}\right.\]
where $x = [\omega]$ and $\deg(y)=2n$.
In the above, $\ut$ is the class in Lemma~\ref{ut}.
\end{proposition}

\begin{proof}
By Proposition~\ref{exy} and Theorem~\ref{c1}, the action is semifree, $\dim(X)=\dim(Y)=2n >2$,
and $$\phi(Y)-\phi(X)=1.$$

First, we compute $c^{S^1}(M)$ and $c(M)$. We have 
$$c^{S^1}(M)|_X =c(X)c^{S^1}(N_X)\,\,\,\mbox{and}\,\,\, c^{S^1}(M)|_Y = c(Y)c^{S^1}(N_Y).$$
By Propositions~ \ref{whole}, \ref{exy}, Lemmas~\ref{chern} and \ref{ut}, we have
$$c^{S^1}(M)|_X = (1+u)^{n+1}\frac{(1+t+u)^{n+1}}{1+t+2u} = (1+\ut)^{n+1}\frac{(1+t+\ut)^{n+1}}{1+t+2\ut}|_X,$$
and
$$c^{S^1}(M)|_Y = (1+v)^{n+1}\frac{(1-t+v)^{n+1}}{1-t+2v} = (1+\ut)^{n+1}\frac{(1+t+\ut)^{n+1}}{1+t+2\ut}|_Y.$$
Since the action is semifree, the map
\begin{equation}\label{inj}
H^*_{S^1}(M; \Z)\to H^*_{S^1}(M^{S^1}; \Z)
\end{equation}
induced by the inclusion is injective. Hence
$c^{S^1}(M)$ is as claimed. Since by Lemma~\ref{rxy},  $H^*(M^{S^1}; \Z)$ is torsion free, we have 
\begin{equation}\label{surj}
 H^*(M; \Z) = H^*_{S^1}(M; \Z)/(t).
\end{equation} 
 So $c(M) = c^{S^1}(M)/(t)$ and it is as claimed. For the  injectivity of (\ref{inj}) and the claim (\ref{surj}),  
we refer to \cite{K} and \cite[Sec. 2]{LT}.

Next, we find a basis for $H^*_{S^1}(M; \Z)$ as an $H^*(\CP^{\infty}; \Z)$-module.
First, by Lemma~\ref{rxy}, 
$$H^*(X; \Z)=\Z[u]/u^{n+1}\,\,\,\mbox{and}\,\,\, H^*(Y; \Z)=\Z[v]/v^{n+1}.$$
By the method given in \cite{K, TW}, a way of getting a basis of $H^*_{S^1}(M; \Z)$ as an $H^*(\CP^{\infty}; \Z)$-module is as follows.
For the basis $\{1, u, \cdots, u^n\}$ of $H^*(X; \Z)$ and the basis $\{1, v, \cdots, v^n\}$ of 
$H^*(Y; \Z)$,  there exists a basis of $H^*_{S^1}(M; \Z)$ as an $H^*(\CP^{\infty}; \Z)$-module:
$$\alpha_i\in H^{2i}_{S^1}(M; \Z)\,\,\mbox{such that}\,\, \alpha_i|_X=u^i, \,\,\mbox{and}$$
$$\beta_i\in H^{2n+2i}_{S^1}(M; \Z)\,\,\mbox{such that}\,\, \beta_i|_X=0,\,\,\mbox{and}\,\,\,\beta_i|_Y=v^ie^{S^1}(N_Y),\,\,\mbox{where $0\leq i\leq n$}.$$
Since $\ut^i |_X = u^i$, we may take $\alpha_i = \ut^i$, $\,\forall\,\,  0\leq i\leq n$.  Next, we find the $\beta_i$'s.
By Proposition~\ref{exy}, 
$$e^{S^1}(N_Y) =\frac{(-t +v)^{n+1}}{-t+2v}.$$ 
We have $(\ut +t)^i|_Y = v^i$. Notice that for $\forall\,\, 0\leq i\leq n$, 
$$\big(\ut^{n+1}(\ut +t)^i\big)|_X = \big((2\ut +t) \beta_i\big) |_X = 0, \,\,\,\mbox{and}\,\,\,\, \big(\ut^{n+1}(\ut +t)^i\big)|_Y=\big( (2\ut + t) \beta_i\big) |_Y.$$
By the injectivity of (\ref{inj}), we get
\begin{equation}\label{beta0}
\ut^{n+1}(\ut +t)^i = (2\ut + t) \beta_i, \,\,\forall\,\,0\leq i\leq n.
\end{equation}
Hence we can express the $\beta_i$'s  as claimed. 

Finally, we will find the ring $H^*(M; \Z)$.  By (\ref{surj}), the image of the basis $\{\alpha_i, \beta_i \,|\, 0\leq i\leq n\}$
under the restriction map
$$r\colon H^*_{S^1}(M; \Z)\to H^*(M; \Z)$$
is  a basis of $H^*(M; \Z)$. We have  
$$r(\alpha_i)=[\omega]^i = x^i, \,\,\forall\,\, 0\leq i\leq n, \,\,\mbox{where $x=[\omega]$}.$$
Let 
$$r(\beta_0)= y.$$ 
Then a basis of $H^*(M; \Z)$ is 
$$\{1, x, \cdots, x^n, y, xy, x^2y, \cdots, x^ny\}, \,\,\mbox{where $\deg(y)=2n$}.$$
Applying the map  $r$ on both sides of  (\ref{beta0}) for $i=0$, we get the relation
$$x^{n+1} = 2xy.$$
We still need to find the relation between $y^2$ and the top generator $x^n y$. For this, we use Theorem~\ref{AB.BV} to integrate $\beta_0^2$ on $M$:
$$\int_M \beta_0^2 = \int_M y^2 = \int_X \frac{\beta_0^2|_X}{e^{S^1}(N_X)} +  \int_Y \frac{\beta_0^2|_Y}{e^{S^1}(N_Y)} = 0 + \int_Y e^{S^1}(N_Y) $$
$$= (-1)^n \int_Y \frac{(t-v)^{n+1}}{t-2v}=(-1)^n \frac{1}{t}\int_Y (t-v)^{n+1}\Big(1+\frac{2v}{t}+\cdots +\big(\frac{2v}{t}\big)^n \Big)$$
$$=(-1)^n \frac{1}{t}\int_Y \Big(t^{n+1}-{n+1\choose 1}t^n v +\cdots + {n+1\choose n}(-1)^n t v^n\Big)\Big(1+\frac{2v}{t}+\cdots +\big(\frac{2v}{t}\big)^n \Big)$$
$$=\frac{1+(-1)^n}{2}.$$
 Hence
\[ H^*(M; \Z) = \left\{\begin{array}{ll}
 \Z[x, y]/(x^{n+1}-2xy, y^2),  & \,\,\mbox{if $n$ is odd},\\
\Z[x, y]/(x^{n+1}-2xy, y^2-x^ny), & \,\,\mbox{if $n$ is even}.
\end{array}\right.\]
\end{proof}

Now we can complete the proof of Theorem~\ref{lg}:
\begin{proof}[Proof of Theorem~\ref{lg}]
$(1)$ and $(2)$ follow from Proposition~\ref{exy}, $(3)$ follows from Lemma~\ref{rxy}, $(4)$ and $(5)$ follow from Proposition~\ref{equibase}, and $(6)$ follows from
Lemma~\ref{nxy} and Proposition~\ref{whole}.
\end{proof}

\section{the ring $H^*(M; \Z)$ determines the other data --- proof of Theorem~\ref{gl}}\label{sec-thm3}

In this section, we prove Theorem~\ref{gl}.

\begin{proposition}\label{rg-rl}
Let $(M, \omega)$ be a compact symplectic manifold admitting a Hamiltonian $S^1$ action such that $M^{S^1}$ consists of two connected components $X$ and $Y$ with $\dim(X)+\dim(Y)=\dim(M)$. If  $H^*(M; \Z)\cong H^*\big(\Gt_2(\R^{2n+2}); \Z\big)$ as rings, then $H^*(X; \Z)\cong H^*(\CP^n; \Z)$ and 
$H^*(Y; \Z)\cong H^*(\CP^n; \Z)$ as rings, with the natural induced symplectic orientations on $X$ and $Y$.
\end{proposition}

\begin{proof}
Since $H^*(M; \Z)\cong H^*\big(\Gt_2(\R^{2n+2}); \Z\big)$ (as groups), and $M$ is a compact oriented manifold, 
we have $\dim(M)=4n$,  and $b_{2i}(M)=1$ if $0\leq 2i \leq 4n$ and $2i\neq 2n$, and $b_{2n}(M)=2$.

Let $\phi$ be the moment map and assume $\phi(X) < \phi(Y)$.
Assume $\dim(X) > \dim(Y)$, then $\dim(X) > 2n$, so $b_{\dim(X)}(M)=1$ by the first paragraph.  But, by (\ref{perfect}) and the fact $\codim(Y)=\dim(X)$, we have
$b_{\dim(X)}(M) =b_{\dim(X)}(X) + 1 = 2$, a contradiction.
So $\dim(X)\leq\dim(Y)$. Similarly, argue using $-\phi$, we have $\dim(Y)\leq\dim(X)$. Hence $\dim(X)=\dim(Y)=2n$.

Using $\phi$ as a Morse-Bott function, since $\codim(Y)=\dim(X)=2n$, the natural restriction map
$$H^i(M; \Z)\to H^i(X; \Z)$$
is an isomorphism for all $0\leq i < 2n$.
If $[\omega]$ is primitive integral, then by assumption, $H^*(M; \Z)\cong H^*\big(\Gt_2(\R^{2n+2}); \Z\big)$ has
a subring $\Z[[\omega]]/[\omega]^{n+1}$ in degrees less than $2n+1$. By the above isomorphisms, if $u=[\omega|_X]$, then the degree less
than $2n$ terms of $H^*(X; \Z)$ has the ring structure $\Z[u]/u^n$. If $u^n$ is not a generator of $H^{2n}(X; \Z)$, then $a u^n$, where $0< a < 1$, is a generator. Since $u\in H^2(X; \Z)$ is a generator, by Poincar\'e duality, 
 $a u^{n-1}\in H^{2n-2}(X; \Z)$ is a generator, a contradiction. Hence $H^*(X; \Z) = \Z[u]/u^{n+1}$.  Similarly, using $-\phi$, we get that $Y$ has the integral cohomology ring of  $\CP^n$.
\end{proof}

Now we can prove Theorem~\ref{gl}.

\begin{proof}[Proof of Theorem~\ref{gl}]
We may assume that $[\omega]$ is a primitive integral class. Let $u=[\omega|_X]$ and $v=[\omega|_Y]$.
By Proposition~\ref{rg-rl}, the assumption implies that 
$$H^*(X; \Z) = \Z[u]/u^{n+1} \,\,\,\mbox{and}\,\,\, H^*(Y; \Z) = \Z[v]/v^{n+1}.$$

 If $\dim(M)=4n=4$, then $X$ and $Y$ are diffeomorphic to $S^2$, so $M$ is diffeomorphic to an $S^2$-bundle over $S^2$. Since $H^*(M; \Z)\cong H^*(\Gt_2(\R^{4}); \Z)$, $M$ is diffeomorphic to $S^2\times S^2$. Since 
$\codim(X)=\codim(Y)=2$ and the action is effective, the action must be semifree. Other claims naturally follow.

For $\dim(M)>4$, the claims follow from Theorem~\ref{lg}.
\end{proof}

\section{when the manifold is K\"ahler --- proof of Theorem~\ref{K}}\label{sec-thm4}

The proof of Theorem~\ref{K} uses Theorems~\ref{c1}, \ref{lg}, \ref{gl} and 
the following result which is part of  \cite[Prop. 4.2]{L}.
 \begin{proposition}\label{equibiho}
 Let $(M, \omega, J)$ be a compact K\"ahler manifold of complex dimension $n$, which admits a holomorphic Hamiltonian circle action. Assume that $[\omega]$ is an integral class.  If $c_1(M)=n[\omega]$, then $M$ is $S^1$-equivariantly biholomorphic to $\Gt_2(\R^{n+2})$, equipped with a standard circle action.
 \end{proposition}

\begin{proof}[Proof of Theorem~\ref{K}]
By Theorems~\ref{c1}, \ref{lg}, and \ref{gl} (and Remark~\ref{d4}),  any of the conditions in Theorem~\ref{K}
implies that $c_1(M)=2n[\omega]$ for a suitable primitive integral K\"ahler class $[\omega]$. Then by Proposition~\ref{equibiho}, $M$ is $S^1$-equivariantly biholomorphic to $\Gt_2(\R^{2n+2})$, equipped with a standard circle action.

Now, let 
$$f\colon (M, \omega, J)\to \big(\Gt_2(\R^{2n+2}), \omega', J'\big)$$
be the $S^1$-equivariant biholomorphism.
By rescaling the symplectic form (in high dimension) or by changing the symplectic class (in dimension $4$), we may assume that $\omega$ and $f^*\omega'$ represent  the same cohomology class on $M$.  Then the family
of forms $\omega_t = (1-t)\omega + t f^*\omega'$ on $M$, where $t\in [0,1]$,  represent the same cohomology class. 
Now we show that each $\omega_t$ is nondegenerate.  For any point $x\in M$, suppose $X\in T_x M$ is such that 
$\omega_t (X, Y) = 0$ for all $Y\in T_x M$. In particular, if $Y = J X$, then $\omega_t (X, J X) = 0$. Using the facts 
$\omega(X, JX)\geq 0$, $f_* (J X) = J' f_* X$, and $\omega'(f_*X, J'f_* X)\geq 0$, we get $X= 0$. 
By Moser's method \cite{Mo}, there exists an $S^1$-equivariant isotopy $\Phi_t$ such that 
$\Phi_t^*\omega_t =\omega$, in particular, $\Phi_1^*f^*\omega' =\omega$. So $f\circ\Phi_1$ is the $S^1$-equivariant symplectomorphism we are looking for.
\end{proof}

\end{document}